\RequirePackage{fix-cm}
\documentclass[smallcondensed]{svjour3}     
\smartqed  
\usepackage{graphicx}
\usepackage{amsmath}
\usepackage{amsfonts}\usepackage{bm}
\usepackage{amssymb}\usepackage{color}\usepackage[margin=1.1in]{geometry}

\numberwithin{equation}{section}
\numberwithin{theorem}{section}
\numberwithin{corollary}{section}
\newtheorem{lem}{Lemma}[section]
\newtheorem{rem}{Remark}[section]
\numberwithin{equation}{section}
\numberwithin{lem}{section}
\numberwithin{rem}{section}

\def\b#1{{\bf#1}}
\begin{document}

\title{
$\b H(\mathrm{curl}^2)$ conforming element for Maxwell's transmission eigenvalue problem using fixed-point approach
}


\author{Jiayu Han \and Zhimin Zhang
       }

\authorrunning{
        \and  \and } 

\institute{           
           Jiayu Han\at
              Beijing Computational Science Research Center, Beijing, 100193, China\\
              School of Mathematical Sciences, Guizhou Normal University, 550001, China\\
              \email{hanjiayu@csrc.ac.cn}
              \and
              Zhimin Zhang\at
              Beijing Computational Science Research Center, Beijing, 100193, China\\
              Department of Mathematics, Wayne State University, Detroit, MI 48202, USA\\
              \email{zmzhang@csrc.ac.cn}
}

\pagestyle{plain} \textwidth 145mm \textheight 215mm \topmargin 0pt

\maketitle
\begin{abstract}
 Using newly developed $\b H(\mathrm{curl}^2)$ conforming elements, we solve the Maxwell's transmission eigenvalue problem.
Both real and complex eigenvalues are considered. Based on the fixed-point weak formulation with reasonable assumptions,
the optimal error estimates for numerical eigenvalues and eigenfunctions (in the ${\bf H}(\mathrm{curl}^2)$-norm and ${\bf H}(\mathrm{curl})$-semi-norm) are established.
Numerical experiments are performed to verify the theoretical assumptions { and} confirm our theoretical analysis.
\end{abstract}

\section{Introduction}

The transmission eigenvalue problem has important applications in the area of inverse scattering{, e.g., simulating non-destructive test} of anisotropic materials.
{ For some background materials such as existence theory, application, and reconstruction of transmission eigenvalues, we refer readers to \cite{cakoni1,cakoni2,cakoni21,cakoni3,colton1} and references therein.
Naturally,  numerical computation of transmission eigenvalues has attracted the attention of scientific community.  There have been some research works on numerical methods for} Helmholtz  transmission eigenvalue problem (HTEP),
{ see, e.g., } \cite{an,colton2,cakoni,ji1,kleefeld1,yang,yang3}.
{ However, numerical} treatment of the Maxwell's transmission eigenvalue problem (MTEP) is relatively rare.
{ An earlier work on the subject can be found} in \cite{monk1} where a {curl-conforming and a mixed finite element were proposed.
The authors reduced the MTEP} to two coupled eigenvalue problems involving  { the second-order curl operator}.
Huang et al. \cite{huang} proposed an eigensolver for computing  a few smallest positive   Maxwell's
transmission eigenvalues.
More recently An and Zhang \cite{an1}  {studied} a spectral method for MTEP on spherical domains and  {obtained} numerical eigenvalues { with} superior accuracy.
The MTEP is { a non-self-adjoint and non-elliptic problem involving the quad-curl operator, which makes the error analysis of its numerical methods difficult} (see the concluding remark in \cite{monk}).
 { There have been some related works on numerical methods for equations with the quad-curl operator
 and the associated eigenvalue problems} \cite{zheng,zhang1,hong,sun1,brenner1,wang}.

In the {  finite element error analysis for HTEP}, the solution operator of its source problem is readily defined to guarantee its compactness in { the} solution space. { However, }
it is difficult to define a { compact} solution operator for { MTEP in $\b H(\mathrm{curl}^2)$. Fortunately,} the fixed-point weak formulation in \cite{cakoni2,cakoni21,cakoni3,paivarinta,sun2} for { MTEP leads} to a
source problem with a well-defined compact solution operator { whose image space is also contained in $\b H(\mathrm{div})$.} The fixed-point weak formulation  is a generalized  eigenvalue problem with the eigenvalue as its { parameter.
 To solve it numerically, an iterative method is usually adopted. An analysis framework of the iterative method for HTEP is well established in \cite{sun} and further developed in \cite{xie}, which motivate us to use it for MTEP}.

Recently  Zhang and Hu et al. \cite{zhang,hu,hu1}   { proposed  $\b H(\mathrm{curl}^2)$-conforming} (or curl-curl conforming) finite elements for solving PDEs with the quad-curl operator.
In this paper, { we use these newly developed $\b H(\mathrm{curl}^2)$-conforming elements to solve MTEP in anisotropic  inhomogeneous medium. Thanks to the conformity of the finite element space, it makes possible to establish convergence theory for the proposed method}.
We first prove the coercivity of bilinear form of the fixed-point weak formulation. { Then} we prove the uniform convergence of discrete operator in $\b H_0(\mathrm{curl}^2,D)$.
Under the assumption on the uniform lower bound { (which can be verified numerically) of the} discrete fixed-point function, the error estimate of discrete eigenvalue is proved using { the} Lagrange mean value theorem.
{ Our analysis also  includes the complex eigenvalue case with the fixed-point weak formulation being modified to guarantee the coercivity of the sesqui-linear} form.

{ To the best of our knowledge, this is the first numerical method with theoretical proof for MTEP with variable coefficients on general polygonal and polyhedral domains.}

  The rest of the paper is organized as follows.
  In Section 2, the fixed-point weak formulation  and its curl-curl conforming element discretization is given then the Maxwell's transmission eigenvalue is expressed as the root of a fixed-point function.
  In Section 3, we discuss the error estimates for real eigenvalues.   The solution operator and some  associated discrete operators are defined and  the compactness of  the solution operator is stated.
  The optimal error estimates are proved using the approximation relations among discrete operators and Babuska-Osborn's theory.
  The error estimates for complex eigenvalues are proved in Section 4.
  Finally,   in Section 5 we present several numerical examples with different indices of fraction to validate the assumption on the uniform lower bound of discrete fixed-point function  and convergence order of curl-curl conforming element.
  The upper boundedness property of the real numerical eigenvalues is also verified in this section.

\section{Preliminaries}
\indent In this paper, we consider the  Maxwell's transmission
eigenvalue problem: Find $k\in \mathbb{C}$, $\bm w, \bm \sigma\in \mathbf L^{2}(D)$,
$\bm w-\bm \sigma\in \mathbf{H}_0(\mathrm{curl}^2,D)$ such that
\begin{eqnarray}\label{ss1}
&& \mathrm{curl}^2 \bm w-k^{2}N\bm w=0,~~~in~ D,\\\label{s2}
 && \mathrm{curl}^2 \bm \sigma-k^{2}\bm \sigma=0,~~~in~ D,\\\label{s3}
 &&\nu\times(\bm w-\bm \sigma)=0,~~~on~ \partial
D,\\\label{s4}
 &&\nu\times\mathrm{curl} (\bm w-\bm \sigma)=0,~~~ on~\partial D,
\end{eqnarray}
where $D\subset \mathbb{R}^{d}$ ($d=2,3$)  is a bounded simply
connected set containing an inhomogeneous medium, and $\nu$ is the
unit outward normal
to $\partial D$. We assume that $N(\bm x)$ is real-valued satisfying $(N(\bm x)-I)^{-1}\in \b W^{1,\infty}(D)$ and
 \begin{align}
   1<N_*\le \overline{\xi}\cdot N(\bm x) \xi\le N^*<\infty,~\|\xi\|=1.
 \end{align}
With obvious changes the analysis approach in this paper is suitable for
 \begin{align}
     \overline{\xi}\cdot N(\bm x) \xi\le N^*<1,~\|\xi\|=1.
 \end{align}
Throughout this paper we adopt the following function spaces
\begin{align*}
&\mathbf{H}(\mathrm{curl},D):=\{ \bm v\in \mathbf L^2(D): \mathrm{curl} \bm v \in\mathbf L^2(D)\},\\
&\mathbf{H}(\mathrm{curl}^s,D):=\{ \bm v\in \mathbf L^2(D): \mathrm{curl}^j \bm v \in\mathbf L^2(D),1\le j\le s\},
\end{align*}
equipped with the norms $\|\cdot\|_{1,\mathrm{curl}}$ and $\|\cdot\|_{s,\mathrm{curl}}$, respectively,
\begin{align*}
  &\mathbf{H}_0(\mathrm{curl}^s,D):=\{ \bm v\in \mathbf L^2(D): \mathrm{curl}^j \bm v \in\mathbf L^2(D),\mathrm{curl}^{j-1}\bm v\times \bm n|_{\partial D}=0,1\le j\le s\},\\
  &\b{H}(\mathrm{div},D):=\{\bm v\in \mathbf L^2(D):~ \mathrm{div} \bm v \in \b L^2(D)\}.
\end{align*}

\indent From \cite{monk1} we know that for $\bm u=\bm w-\bm \sigma\in
\mathbf{H}_0(\mathrm{curl}^2,D)$,  the weak formulation for the transmission eigenvalue
problem (\ref{ss1})-(\ref{s4}) can be stated as follows: Find $k\in\mathbb C$ and $\bm u \in\mathbf{H}_0(\mathrm{curl}^2,D)$ such that
\begin{eqnarray}\label{2.5}
\left((N-I)^{-1}(\mathrm{curl}^2 \bm u-k^{2} \bm u),\mathrm{curl}^2 \bm v-\overline{k^{2}} N\bm v
\right)=0,~~~\forall \bm v \in\mathbf{H}_0(\mathrm{curl}^2,D),
\end{eqnarray}
Let $\tau=k^2$ as usual. Following the treatment approach in \cite{cakoni2,cakoni21,cakoni3,paivarinta,sun,sun2}, we
consider the  eigenvalue problem in the weak form
\begin{align}\label{1.6}
 \mathcal A_\tau(\bm u,\bm v)=\tau \mathcal B(\bm u,\bm v),~\forall \bm v \in\mathbf{H}_0(\mathrm{curl}^2,D)
\end{align}
where $$ \mathcal A_\tau(\bm u,\bm v)=\left((N-I)^{-1}(\mathrm{curl}^2 \bm u-\tau \bm u),\mathrm{curl}^2 \bm v-\overline{\tau} \bm v
\right)+\tau^2(\bm u,\bm v),$$
$$ \mathcal B(\bm u,\bm v)=(\mathrm{curl} \bm u,\mathrm{curl} \bm v).$$
We will consider the edge element approximations based on the  weak formulation (\ref{1.6}).
The authors in \cite{hu} propose three families of curl-curl conforming elements. For simplicity in this paper we merely show the lowest order element in \cite{hu} and the second family in  \cite{zhang}.
Let $\pi_h$ be a  regular triangulation of $D$ composed of the elements ${\kappa}$.
The curl-curl conforming  edge element  \cite{zhang} generates the spaces
 $$\mathbf{U}_{h}=\{\bm v_h\in\mathbf H_0(\mathrm{curl}^2,D):\bm v_h|_\kappa\in \b P_l(\kappa)\bigoplus\{ \bm p\in\b{\widetilde{P}}_{l+1}(\kappa):\bm{x}\cdot\bm p=0\},~\forall \kappa\in\pi_h\},$$
where $\b P_l(\kappa)$ is the polynomial space of  degree less than or equal to $l$($l\ge3$) on $\kappa$ and $\b{\widetilde{P}}_{l+1}(\kappa)$ is the
homogeneous polynomial space of  degree $l+1$  on $\kappa$.
The lowest order element  \cite{hu} generates the spaces
 $$\mathbf{U}_{h}=\{\bm v_h\in\mathbf H_0(\mathrm{curl}^2,D):\bm v_h|_\kappa\in \nabla P_1(\kappa)\bigoplus \mathfrak p W_1(\kappa),~\forall \kappa\in\pi_h\},$$
 where $W_1(\kappa):=P_1(\kappa)\bigoplus span\{\lambda_1\lambda_2\lambda_3\}$,
 $\mathfrak p  v:=\int_0^1 t\bm x^\bot v(t\bm x)dt$ and $\bm x^\bot=(-x_2,x_1)^T$.
Adopting the curl-curl conforming element, we give the discrete form of the Maxwell's transmission eigenvalue problem
\begin{align}\label{1.7}
 \mathcal A_{\tau_h}(\bm u_h,\bm v_h)=\tau_h \mathcal B(\bm u_h,\bm v_h),~\forall \bm v _h \in~\mathbf{U}_h.
\end{align}
Now we consider the following generalized eigenvalue problem:
\begin{align}\label{1.8}
 \mathcal A_{\tau}(\bm u,\bm v)=\lambda(\tau) \mathcal B(\bm u,\bm v),~\forall \bm v  \in\mathbf{H}_0(\mathrm{curl}^2,D)
\end{align}
and its discrete form
\begin{align}\label{1.9}
 \mathcal A_{\tau}(\bm u_h,\bm v_h)=\lambda_h(\tau) \mathcal B(\bm u_h,\bm v_h),~\forall \bm v _h \in\mathbf{U}_h.
\end{align}
Then $\lambda(\tau)$ and $\lambda_h(\tau)$ is a continuous function of $\tau$. From (\ref{1.6}), the Maxwell's transmission eigenvalue is the root of
$$f(\tau):=\lambda(\tau)-\tau$$
while the discrete eigenvalue in  (\ref{1.7}) is the root of
$$f_h(\tau):=\lambda_h(\tau)-\tau.$$

We need the following error estimates for the curl-curl element interpolation.

\begin{lem}[Theorem 3.4 in \cite{zhang} or Theorem 5.1 in \cite{hu}] If $\bm v\in \bm H^{s-1}(\kappa)$ and $\mathrm{curl}\bm v\in \bm H^s(\kappa),1+\delta\le s\le l+1$ with $\delta>0$ then there hold
the following error estimates for the finite element interpolation $I_h$
\begin{align}\label{3.10}
&\|\bm v-I_h\bm v\|_{0,\kappa}+\|\mathrm{curl}(\bm v-I_h\bm v)\|_{0,\kappa}\lesssim h_\kappa^{s-1}\|\bm v\|_{s-1,\kappa}+h_\kappa^{s}\|\mathrm{curl}\bm v\|_{s,\kappa},\\
&\|\mathrm{curl}^2(\bm v-I_h\bm v)\|_{0,\kappa}\lesssim h_\kappa^{s-1}\|\mathrm{curl}\bm v\|_{s,\kappa},\label{3.11}
\end{align}
where
the symbols $a \lesssim b$ and $a \gtrsim b$    mean that $a \le Cb$ and $a \ge Cb$ respectively, and $C$ denotes a positive constant independent of   mesh parameters
and may not be the same   in different places. For the lowest order curl-curl conforming element, the above estimates are valid with $s=2$.  
\end{lem}
\section{Error estimates for real eigenvalues}
In  this section we let the eigenvalue $\tau\ne0$ in (\ref{1.6}) be a real number.
The following lemma provides useful properties of the generalized eigenvalue problems.
\begin{lem}
$\mathcal A_\tau$ is a coercive sesquilinear form on $\mathbf{H}_0(\mathrm{curl}^2,D)$.
\end{lem}
\begin{proof}Pick up any $\bm u\in\mathbf{H}_0(\mathrm{curl}^2,D)$. We have
\begin{align*}
|\mathcal A_\tau(\bm u,\bm u)|&\ge \gamma\|\mathrm{curl}^2\bm u-\tau \bm u\|^2+\tau^2\|\bm u\|^2\\
                 &\ge\gamma\|\mathrm{curl}^2\bm u\|^2-2\gamma\tau \|\mathrm{curl}^2\bm u\|\|\bm u\|+(\tau^2\gamma+\tau^2)\|\bm u\|^2\\
                 &=\epsilon(\tau \|\bm u\|-\frac\gamma\epsilon\|\mathrm{curl}^2\bm u\|)^2+(\gamma-\frac{\gamma^2}\epsilon)\|\mathrm{curl}^2\bm u\|^2+(1+\gamma-\epsilon)\tau^2\|\bm u\|^2
\end{align*}
where $\gamma=(N^*-1)^{-1}$ and $\gamma<\epsilon<\gamma+1$. Thanks to the result of Lemma 2.1 in \cite{hong}
\begin{align*}
2\|\mathrm{curl}\bm u\|\le\| u\|
                 +\|\mathrm{curl}^2\bm u\|
\end{align*}
the assertion is valid.
\end{proof}

Based on Lemma 3.1, we can define the solution operator $T_\tau\bm f\in \mathbf{H}_0(\mathrm{curl}^2,D)$:
\begin{align}\label{1.10s}
 \mathcal A_{\tau}(T_\tau\bm f,\bm v)=\mathcal B(\bm f,\bm v),~\forall \bm v  \in\mathbf{H}_0(\mathrm{curl}^2,D)
\end{align}
and its discrete operator
$T_{\tau,h}\bm f\in \mathbf{U}_h$:
\begin{align}\label{3.2}
 \mathcal A_{\tau}(T_{\tau,h}\bm f,\bm v)=\mathcal B(\bm f,\bm v),~\forall \bm v  \in\mathbf{U}_h.
\end{align}
Next we shall analyze the compactness of the operator $T_{\tau}$ for $\tau\ne0$. It is obvious from Lemma 3.1
\begin{align}\label{1.12s}
 \|T_\tau\bm f\|_{2,\mathrm{curl}}\lesssim\|\mathrm{curl}\bm f\|_{},~~\forall \bm f\in\{\bm v: \mathrm{curl}\bm v\in L^2(D)\}.
\end{align}

\begin{lem}
$T_{\tau}$($\tau\ne0$) is compact in  $\b H_0(\mathrm{curl}^2,D)$.
\end{lem}
\begin{proof}Let $\{\bm v_i\}_{i=1}^{\infty}$ be a sequence in $\b H_0(\mathrm{curl}^2,D)$ with $\|\bm v_i\|_{2,\mathrm{curl}}\le1$. Then $\{\mathrm{curl}\bm v_i\}_{i=1}^{\infty}\subset \b H_0(\mathrm{curl},D)\cap \b H(\mathrm{div},D)$ have a convergence subsequence in $\b L^2(D)$, still denoted by itself. Thanks to (\ref{1.12s}), $\{T_\tau\bm v_i\}_{i=1}^{\infty}$ is cauchy in $\b H_0(\mathrm{curl}^2,D)$ and so have a convergence point therein. This leads to the compactness of $T_\tau$.
\end{proof}
According to Lemmas 3.1 and 2.1 we have the error estimate for the discrete problem (\ref{3.2})
\begin{align}\label{1.12}
 \|(T_{\tau,h}-T_\tau)\bm f\|_{2,\mathrm{curl}}\lesssim\inf_{\bm v\in\b H_h}\|T_\tau\bm f-\bm v\|_{2,\mathrm{curl}}\lesssim  h^{r-1}(\|T_\tau\bm f\|_{r-1}+\|\mathrm{curl}T_\tau\bm f\|_{r}),~r>1.
\end{align}
We define the Ritz projection $P_{\tau,h}:\b H_0(\mathrm{curl},D)\rightarrow \b U_h$ such that 
\begin{align}
 \mathcal A_\tau(P_{\tau,h}\bm w-\bm w,\bm v)=0,~~\forall \bm v  \in\b U_h.
\end{align}
We have the relation $T_{\tau,h}=P_{\tau,h}T_{\tau}$ between $T_{\tau}$ and $T_{\tau,h}$.
This leads to
\begin{align}\label{1.16}
 \|T_{\tau,h}-T_\tau\|_{2,\mathrm{curl}}=\sup_{\bm f\in \mathbf{H}_0(\mathrm{curl}^2,D)\atop\|\bm f\|_{2,\mathrm{curl}}=1}   \|(I-P_{\tau,h})T_\tau\bm f\|_{2,\mathrm{curl}}\rightarrow0
\end{align}
 where we have used $T_\tau\{\bm f\in \mathbf{H}_0(\mathrm{curl}^2,D):\|\bm f\|_{2,\mathrm{curl}}=1\}$ is a relatively compact set in $\mathbf{H}_0(\mathrm{curl}^2,D)$.


Using the spectral approximation theory \cite{babuska},
we are in a position to establish the following \textit{a} \textit{priori} error estimates for  the finite element approximation (\ref{1.9}).
\begin{lem}
Let $(\lambda_h(\tau),\bm u_{h}(\tau))$ with $\mathcal B(\bm u_{h}(\tau),\bm u_{h}(\tau))=1$ be an eigenvalue of the  problem (\ref{1.9}) that converges to $(\lambda(\tau),\bm u_{}(\tau))$  and the eigenfunction space $M(\lambda(\tau)$ satisfies $M(\lambda(\tau))\subset \b H^{s-1}(D)$, $\mathrm{curl}\left(M(\lambda(\tau))\right)$ $\subset \b H^s(D)$ with $l+1\ge s>1$ then
\begin{eqnarray}\label{2.34t}
&&\|\bm u_{h}(\tau)-\bm u_{}(\tau)\|_{2,\mathrm{curl}}\lesssim
h^{s-1} ,\\
\label{2.35ss}
&&|\lambda(\tau)-\lambda_{h}(\tau)|\lesssim h^{2(s-1)}.
\end{eqnarray}
\end{lem}

Assume $\tau_h\to\tau^*$. In order to prove the approximation relation between $\lambda_h(\tau_h)$ and $\lambda(\tau_h)$ we introduce an auxiliary    operator
$T_{\tau_h,h}\bm f\in \mathbf{U}_h$:
\begin{align}
 \mathcal A_{\tau_h}(T_{\tau_h,h}\bm f,\bm v)=\mathcal B(\bm f,\bm v),~\forall \bm v  \in \mathbf{U}_h.
\end{align}
Then $\forall \bm v  \in\mathbf{U}_h$
\begin{align*}
\mathcal A_{\tau^*}(T_{\tau_h,h}\bm f-T_{\tau^*,h}\bm f,\bm v)&=
\mathcal A_{\tau^*}(T_{\tau_h,h}\bm f,\bm v)- \mathcal A_{\tau_h}(T_{\tau_h,h}\bm f,\bm v)\\
&\lesssim  |\tau_h-\tau^*|\|T_{\tau_h,h}\bm f\|_{2,\mathrm{curl}}\|\bm v\|_{2,\mathrm{curl}}\\
&\lesssim  |\tau_h-\tau^*|\|\mathrm{curl}\bm f\|_{}\|\bm v\|_{2,\mathrm{curl}}.
\end{align*}
Hence
\begin{align}\label{1.22}
\|T_{\tau_h,h}-T_{\tau^*,h}\|_{2,\mathrm{curl}}\lesssim\sup_{\bm f\in \mathbf{H}_0(\mathrm{curl}^2,D)\atop\|\bm f\|_{2,\mathrm{curl}}=1}\sup_{\bm v\in \mathbf{U}_h\atop\|\bm v\|_{2,\mathrm{curl}}=1}\mathcal A_{\tau^*}(T_{\tau_h,h}\bm f-T_{\tau^*,h}\bm f,\bm v)\lesssim  |\tau_h-\tau^*|\rightarrow0.
\end{align}
Note that
\begin{align}\label{1.16}
 \|T_{\tau^*,h}-T_{\tau^*}\|_{2,\mathrm{curl}}= \|(P_{\tau^*,h}-I)T_{\tau^*}\|_{2,\mathrm{curl}}
 \lesssim
 \sup_{\bm f\in \mathbf{H}_0(\mathrm{curl}^2,D)\atop\|\bm f\|_{2,\mathrm{curl}}=1}   \|(P_{\tau^*,h}-I)T_{\tau^*}\bm f\|_{2,\mathrm{curl}}\rightarrow0.
\end{align}
The above two uniform convergence results give
\begin{align}\label{3.14s}
\|T_{\tau_h,h}-T_{\tau^*}\|_{2,\mathrm{curl}}\rightarrow0.
\end{align}
Meanwhile the similar argument to show (\ref{1.22})  can derive
\begin{align}\label{1.22t}
\|T_{\tau_h}-T_{\tau^*}\|_{2,\mathrm{curl}}\lesssim|\tau_h-\tau^*|\rightarrow0.
\end{align}
Using the standard spectral approximation theory in \cite{babuska,chatelin} then in virtue of  (\ref{3.14s})
and (\ref{1.22t})  we have
\begin{lem}
  Assume $\tau_h\to\tau^*$.  Let $\lambda_h(\tau^*)$, $\lambda(\tau^*)$ and   $M(\lambda(\tau^*))$ be as in Lemma 3.3 then
\begin{eqnarray}
\label{2.35}
|\lambda(\tau_h)-\lambda(\tau^*)|+|\lambda_{h}(\tau_h)-\lambda(\tau^*)|
\lesssim h^{2(s-1)}+|\tau_h-\tau^*|\to0.
\end{eqnarray}
\end{lem}

In order to study the convergence of discrete eigenvalues in a bounded interval, we have to verify
their boundedness. The following result,  a direct citation of Theorem 3.3 in \cite{cakoni3}, is given without proof.
\begin{lem}Let $\widetilde\tau_{h}$ and $\widehat\tau_h$ be the eigenvalue of (\ref{1.7}) with $N:=N_*I$ and
$N:=N^*I$, respectively. Then there is an eigenvalue $\tau_{h}$ of (\ref{1.7}) such that
$\widehat\tau_h\le\tau_h\le \widetilde\tau_{h}$.
\end{lem}
With the aid of standard error analysis of FEM for quad-curl eigenvalue problem with constant coefficients, it is somewhat easier to prove $\widehat\tau_h$ and $\widetilde\tau_h$ converges to the eigenvalue $\widehat\tau$ of (\ref{1.6}) with   $N:=N^*I$ and the eigenvalue $\widetilde\tau$ with
$N:=N_*I$, respectively. Then it follows by Lemma 3.4 that $\widehat\tau\le\tau_h\le \widetilde\tau_{}$ for $h$ small enough.


%
%
%


\begin{theorem}
Let $f_h(\tau)\in C^1([a,b])$ with $a>0$ and $f(\tau)$ be two continuous functions on  $[a,b]$.
Let $\{\tau_{h_i}\}_1^{\infty}\subset[a,b]$ satisfy $f_{h_i}(\tau_{h_i})=0$ then there is $\tau^*$ such that
$ f(\tau^*)=0$ and a subsequence  $\tau_{h_i}\to\tau^*$ ($i\to\infty$). We adopt the following assumption.

{\bf Assumption A.} There is a positive constant $c$ such that $\min_{\tau\in [a,b]}| f'_{h}(\tau)|\ge c$ for $h$ small enough.

Then it holds
\begin{align}\label{s1.16}
 |\tau_{h_i}-\tau^*|\le|f_{h_i}(\tau^*)-f(\tau^*)| {\overline c},~for~i~large~enough.
\end{align}

Conversely, let the interval $[a,b]$ be such that $\tau^*\in[a,b]$ with $f(\tau^*)=0$ and   $\{\tau_{h}: f_{h}(\tau_{h})=0,\forall h<\delta\}\subset [a,b]$ for a small $\delta>0$. If Assumption A  is valid
then any sequence $\{\tau_{h_i}\}_1^{\infty}\subset[a,b]$ satisfying $f_{h_i}(\tau_{h_i})=0$ will converge to $\tau^*$  and the above (\ref{s1.16}) is valid.
\end{theorem}
\begin{proof}
Let $\{\tau_{h_i}\}_1^{\infty}\subset[a,b]$ satisfy $f_{h_i}(\tau_{h_i})=0$.
Note that the sequence $\{\tau_{h_i}\}_1^{\infty}\subset [a,b]$ does  have a cluster point, i.e., there is a subsequence, still denoted by itself, 
converging to $\tau^*\in [a,b]$. In virtue of Lemma 3.4
we have
\begin{align*}
  f(\tau_{h_i})\rightarrow 0,~i\rightarrow\infty.
\end{align*}
 Hence due to the continuity of $f$ we have
\begin{align*}
  f(\tau^*)= 0.
\end{align*}
Let Assumption A hold true.
Using Lagrange mean value theorem we have
\begin{align*}
  f_{h_i}(\tau^*)-f_{h_i}(\tau_{h_i})= f'_{h_i}(\psi)(\tau^*-\tau_{h_i}), ~ \psi~is~between~\tau_{h_i}~and~
  \tau^*
\end{align*}
that is
\begin{align}\label{t}
  f_{h_i}(\tau^*)-f(\tau^{*})=  f'_{h_i}(\psi)(\tau^*-\tau_{h_i}).
\end{align}
Then (\ref{s1.16})  follows.

Conversely, let $[a,b]$ be such that $\tau^*\in[a,b]$ with $f(\tau^*)=0$ and any sequence $\{\tau_{h_i}\}$ with $h_i<\delta$   falls in $[a,b]$. Let the modified Assumption A be valid. We give the proof  by contradiction. Suppose  there is a $\varepsilon>0$  so that for any fixed positive $\widetilde\delta<\delta$  if $h_i<\widetilde\delta$ then $|\tau_{h_i}-\tau^*|>\varepsilon$.
We modify the expansion (\ref{t}) as
\begin{align}\label{t1}
  f(\tau^{*})-f_{h_i}(\tau^*)= f'_{h_i}(\psi)(\tau_{h_i}-\tau^*).
\end{align}
This leads to a contraction by letting $i\rightarrow\infty$ and using Lemma 3.3.
 Hence $\tau_{h_i}\rightarrow\tau^*$ ($i\rightarrow\infty$). Then (\ref{s1.16}) follows by using (\ref{t}) again.

\end{proof}

\begin{rem} In Theorem 3.1 the condition $\min_{\tau\in[a,b]}|f'_{h}(\tau)|\ge c$ ($h$ small enough) can be reduced into $|f'_{h}(\tau^*)|$ $\ge c$ ($h$ small enough).
We state the practicality of the conditions given in above theorem. The assumption $\{\tau_{h_i}\}_i^{\infty}\subset[a,b]$ is usually satisfied in the case that $\{\tau_{h_i}\}_i^{\infty}$ converges to a point $\tau^*$.  Assumption A can be verified
  when $f^{}_{h_i}(\tau)$ is a strictly monotonic function sequence   in a small neighbourhood $[a,b]$ of  $\tau_{h_i}$.
  In addition, we can prove  $f'_{h_i}(\tau^*)\to f'_{}(\tau^*)$ ($i\rightarrow\infty$). In fact, differenating on both sides of (\ref{1.9})
  \begin{align}\label{3.201}
 \mathcal A'_{\tau^*}(\bm u_h(\tau^*),\bm v_h)+\mathcal A_{\tau^*}(\bm u'_h(\tau^*),\bm v_h)=\lambda'_h(\tau^*) \mathcal B(\bm u_h(\tau^*),\bm v_h)+\lambda_h(\tau^*) \mathcal B(\bm u'_h(\tau^*),\bm v_h),~\forall \bm v _h \in\mathbf{U}_h.
\end{align}
Taking $\bm v_h:=\bm u_h(\tau^*)$ with $\mathcal B(\bm u_h(\tau^*),\bm u_h(\tau^*))=1$ we have
 \begin{align}\label{3.221}
   \lambda'_h(\tau^*)&=\mathcal A'_{\tau^*}(\bm u_h(\tau^*),\bm u_h(\tau^*))\nonumber\\
   &=-\left((N-I)^{-1}\mathrm{curl}^2 \bm u_h(\tau^*), \bm u_h(\tau^*)
\right)-\left((N-I)^{-1} \bm u_h,\mathrm{curl}^2 \bm u_h(\tau^*)
\right)\nonumber\\
&~~~~+2\tau^*(N(N-I)^{-1}\bm u_h(\tau^*),\bm u_h(\tau^*)).
 \end{align}
Similarly we have
 \begin{align}\label{3.222t}
   \lambda'(\tau^*)&=-\left((N-I)^{-1}\mathrm{curl}^2 \bm u(\tau^*), \bm u(\tau)
\right)-\left((N-I)^{-1} \bm u(\tau^*),\mathrm{curl}^2 \bm u(\tau^*)
\right)\nonumber\\
&~~~~+2\tau^*(N(N-I)^{-1}\bm u(\tau^*),\bm u(\tau^*))
 \end{align}
 with the eigenfunction satisfying $\mathcal B(\bm u,\bm u)=1$.
So the assertion is true due to (\ref{2.34t}).
Hence   $\lambda'(\tau^*)\ne1$ implies $|f'_{h_i}(\tau^*)|$ $\ge c$ ($i$ small enough).
\end{rem}

\begin{theorem}
Assume $\tau_h\to\tau^*\in(a,b)$ with $a>0$.  Let $\lambda_h(\tau),\lambda(\tau)\in C^1([a,b])$.   Then
\begin{eqnarray}
\label{2.36s}
&&|\lambda'_{h}(\tau_h)-\lambda'(\tau^*)|+|\lambda'_{h}(\tau^*)-\lambda'(\tau^*)|\to0.
\end{eqnarray}
\end{theorem}
\begin{proof}$\lambda_{}(\tau)\in C^1([a,b])$ implies that  $\lambda'_{}(\tau)$ has uniform continuity on $[a,b]$.
For any fixed $\varepsilon>0$, there exists $\delta(\varepsilon)>0$ such that if $|\Delta\tau|<\delta(\varepsilon)$ then
\begin{align}\label{3.121}
  |\lambda'_{}(\tau+\Delta\tau)-\lambda'_{}(\tau)|<\varepsilon/2.
\end{align}
We deduce from (\ref{3.221}) and (\ref{3.222t})
\begin{align}
 \max_{\tau\in [a,b]} |\lambda'_{h}(\tau)-\lambda'_{}(\tau)|\to0.
\end{align}
Hence we know for $h<h_0(\varepsilon,a,b)$ it holds for $\tau\in[a,b]$
\begin{align}\label{3.121t}
  |\lambda'_{h}(\tau)-\lambda'(\tau)|<\varepsilon/2.
\end{align}
For $h<h_1(\delta,a,b)$ small enough it holds $\tau_h\in(a,b)$ and $|\tau_h-\tau^*|<\delta(\varepsilon)$.  Then
  we deduce  from (\ref{3.121}) and (\ref{3.121t}) that for $h<\min(h_0,h_1)$
\begin{align}\label{3.122}
  |\lambda'_{h}(\tau_h)-\lambda'_{}(\tau^*)|\le   |\lambda'_{h}(\tau_h)-\lambda'_{}(\tau_h)|+|\lambda'_{}(\tau_h)-\lambda'(\tau^*)|<\varepsilon.
\end{align}
Namely $|\lambda'_{h}(\tau_h)-\lambda'_{h}(\tau^*)|\to0$ ($h\to0$).
\end{proof}

\begin{rem}
The above theorem implies that if $\tau_h\to\tau^*$ then $f'_h(\tau_h)=\lambda'_{h}(\tau_h)-1$ can be regarded as  an indicator to detect $|f_h'(\tau^*)|=|\lambda'_{h}(\tau^*)-1|$ greater than 0 strictly in Remark 1. Since $\lambda'_{h}(\tau_h)$ is readily computed by (\ref{3.221}), this is more  convenient in practical computation.
\end{rem}

\begin{theorem} Under the assumptions in Theorem 3.1,
    let $\tau_h=\lambda_h(\tau_h)$ be an eigenvalue of the  problem (\ref{1.7}) that converges to $\tau^*=\lambda(\tau^*)$.  Let the eigenfunction space $M(\lambda(\tau^*))$ satisfy $M(\lambda(\tau^*))\subset \b H^{s-1}(D)$, $\mathrm{curl}\left(M(\lambda(\tau^*))\right)\subset \b H^{s}(D)$ with $l+1\ge s>1$.
Let $\bm u_{h}(\tau_h)$ be the corresponding  eigenfunction of the problem (\ref{1.7})
  and $\|\bm u_{h}(\tau_h)\|_{2,\mathrm{curl}}=1$, then there exists an eigenfunction
$\bm u(\tau^*) $ and $s_0>1$ such that

\begin{eqnarray}
&&|\tau^*-\tau_h|\lesssim h^{2(s-1)},\label{1.25}\\\label{2.34}
&&\|\bm u_{h}(\tau_h)-\bm u_{}(\tau^*)\|_{2,\mathrm{curl}}\lesssim
h^{s-1}, \\\label{2.35}
&&\|\mathrm{curl}(\bm u_{h}(\tau_h)-\bm u_{}(\tau^*))\|_{}\lesssim
h^{s-2+s_0}.
\end{eqnarray}
\end{theorem}
\begin{proof}The combination of (\ref{2.35ss}) and (\ref{s1.16})   gives
(\ref{1.25}).
According to  the standard spectral approximation theory in \cite{babuska}, using (\ref{3.14s}), (\ref{1.22}) and the argument in (\ref{1.16}) we have
\begin{align*}
\|\bm u(\tau^*)-\bm u_{h}(\tau_h)\|_{2,\mathrm{curl}}&\lesssim \|(T_{\tau^*}-T_{\tau^*,h}+T_{\tau^*,h}-T_{\tau_h,h})\bm u(\tau^*)\|_{2,\mathrm{curl}}\\
&\lesssim \|(P_{\tau^*,h}-I)\bm u(\tau^*)\|_{2,\mathrm{curl}}+\|(T_{\tau^*,h}-T_{\tau_h,h})\bm u(\tau^*)\|_{2,\mathrm{curl}}\\
&\lesssim  h^{s-1}+|\tau^*-\tau_h|.
\end{align*}
This together with (\ref{1.25})   yields the assertion (\ref{2.34}).
Introduce the auxiliary problem:
Find $\bm\Phi\in \mathbf{H}_0(\mathrm{curl}^2,D)$ such that
\begin{align}\label{2.14ss}
\mathcal A_{\tau^*}(\bm\Phi, \bm v)=\mathcal B\left((T_{\tau^*}-T_{\tau^*,h})\bm f, \bm v\right),~~~\forall
\bm v\in \mathbf{H}_0(\mathrm{curl}^2,D).
\end{align}
We adopt the following a-priori regularity assumption(see (5.7) in \cite{zhang} and Remark 3.5 in \cite{brenner1})
\begin{eqnarray}\label{as}
\|\mathrm{curl}{\bm\Phi}\|_{s_0}+\|\bm\Phi\|_{s_0-1}
\lesssim \|\mathrm{curl}(T_{\tau^*}-T_{\tau^*,h})\bm f\|_{}
\end{eqnarray}
with $s_0>3/2$.
We shall verify this assumption under $N=nI$  with   $n,(n-1)^{-1}\in W^{1,\infty}(D)$, $n(\bm x)>1$ $\forall \bm x\in D$    and the two-dimensional   domain $D$.
Taking $\bm v=\nabla p$ in (\ref{2.14ss}) for any  $p\in H^1_0(D)$, we have
\begin{eqnarray}\label{1.13sss}
     \tau^{*2}\left(n(n-1)^{-1}\bm\Phi,\nabla p\right) - \tau^*\left( (n-1)^{-1}  \mathrm{curl}^2 \bm\Phi, \nabla p\right)  =0.
\end{eqnarray}
It follows that
\begin{align}\label{1.13}
\|\mathrm{div}\left(\frac{n}{n-1}\bm\Phi\right)\|\lesssim  \|\bm\Phi\|_{2,\mathrm{curl}}\lesssim \|\mathrm{curl}(T_{\tau^*}-T_{\tau^*,h})\bm f\|_{},~~\forall \bm f\in\b H_0(\mathrm{curl},D).
\end{align}
Hence $\bm\Phi\in \b H_0(\mathrm{curl}^2,D)\cap\{\bm v:\mathrm{div}\left(\frac{n}{n-1}\bm v\right)\in L^2(D)\}\subset \b H^{s_0-1}(D)$ for some $s_0>3/2$.
 Let
 $y:=\mathrm{curl}\bm\Phi\in H^1_0(D) $ and $\bm g:=(T_{\tau^*}-T_{\tau^*,h})\bm f\in\mathbf{H}(\mathrm{curl},D)$. 
Note that
$$\mathrm{curl}^2((n-1)^{-1}\mathrm{curl}^2\bm\Phi)=\tau ^* (n-1)^{-1}\mathrm{curl}^2\bm\Phi+\tau ^*\mathrm{curl}^2 ((n-1)^{-1}\bm\Phi)-(\tau^*)^2 \bm \Phi+\mathrm{curl}^2\bm g.$$
i.e.,
$$\mathrm{curl}^2((n-1)^{-1}\mathrm{curl} y)=\tau^* n(n-1)^{-1}\mathrm{curl}^2 \bm\Phi+\tau ^*\mathrm{curl}^2 ((n-1)^{-1}\bm\Phi)-(\tau^*)^2 \bm\Phi+\mathrm{curl}^2\bm g.$$
Since $$\mathrm{curl}((n-1)^{-1}\mathrm{curl} y)=-(n-1)^{-1}\Delta  y-\nabla ((n-1)^{-1})\cdot\nabla  y$$ we have
\begin{align*}
-\mathrm{curl}((n-1)^{-1}\Delta y)&=\tau^* (n-1)^{-1}\mathrm{curl}^2\bm\Phi+\tau ^*\mathrm{curl}^2 ((n-1)^{-1}\bm\Phi)-(\tau^*)^2 \bm \Phi\\
&~~~~+\mathrm{curl}(\nabla ((n-1)^{-1})\cdot\nabla  y)+\mathrm{curl}^2\bm g.
\end{align*}
This implies $\nabla((n-1)^{-1}\Delta  y)\in \b H^{-1}(D)$. It is obvious that $(n-1)^{-1}\Delta   y\in  H^{-1}(D)$.
Thanks to Lemma 3.2 in \cite{brenner1} we have  $(n-1)^{-1}\Delta y\in L^2(D)$. 
Due to the regularity estimate of Possion equation, there is a number greater than 3/2, still denoted by $s_0$, such that  $\| y\|_{s_0}\lesssim \|\Delta y\|$. Then  $y\in H^{s_0}(D)$ and we conclude the assertion (\ref{as}).
Let $\bm\Phi_h$ be the finite element interpolation approximation of $\bm\Phi$. For any $\bm f\in \b L^2(D)$ we have from (\ref{2.14ss}) and the interpolation error in Lemma 2.1
\begin{align}
\|\mathrm{curl}(T_{\tau^*}-T_{\tau^*,h})\bm f\|_{}^2&=\mathcal A_{\tau^*}((T_{\tau^*}-T_{\tau^*,h})\bm f,\bm\Phi)\nonumber\\
&=\mathcal A_{\tau^*}((T_{\tau^*}-T_{\tau^*,h})\bm f,\bm\Phi-\bm\Phi_h)\nonumber\\
&\lesssim \|(T_{\tau^*}-T_{\tau^*,h})\bm f\|_{2,\mathrm{curl}}h^{s_0-1}(\|\mathrm{curl}\bm\Phi\|_{s_0}+\|\bm\Phi\|_{s_0-1})\nonumber\\
&\lesssim h^{s_0-1}\|(T_{\tau^*}-T_{\tau^*,h})\bm f\|_{2,\mathrm{curl}}\|\mathrm{curl}(T_{\tau^*}-T_{\tau^*,h})\bm f\|_{}.\label{1.29s}
\end{align}
This implies
$$\|T_{\tau^*}-T_{\tau^*,h}\|_{\dot{\b H}_0(\mathrm{curl})}\to0$$
and
\begin{align}\label{s1}
  \|(T_{\tau^*}-T_{\tau^*,h})|_{M(\lambda)}\|_{\dot{\b H}_0(\mathrm{curl})}\lesssim h^{s_0+s-2}
\end{align}
where $\dot{\b H}_0(\mathrm{curl},D):={\b H}_0(\mathrm{curl},D)/\ker(\mathrm{curl})$.
The similar argument as in (\ref{1.22}) leads to
\begin{align}
\|T_{\tau_h,h}-T_{\tau^*,h}\|_{\dot{\b H}_0(\mathrm{curl})}\lesssim  |\tau_h-\tau^*|\rightarrow0.
\end{align}
Using  Theorem 7.4
in \cite{babuska}, we deduce from the two estimates above
\begin{align*}
\|\mathrm{curl}(\bm u(\tau^*)-\bm u_{h}(\tau_h))\|_{}
 \lesssim\|(T_{\tau_h,h}-T_{\tau^*})\mid_{M(\lambda)}\|_{\dot{\b H}_0(\mathrm{curl})}\lesssim h^{s+s_0-2}+|\tau_h-\tau^*|.
\end{align*}
Using (\ref{1.25}), this gives (\ref{2.35}).
\end{proof}

\section{Error estimates for complex eigenvalues}
In  this section we assume the eigenvalue $\tau:=\tau_1+\tau_2i$ in (\ref{1.6}) be a complex number with $\tau_2\ne0$, $|\tau_1|>(\sqrt2-1)|\tau_2|$ and $N=nI$.
Given a positive number $\eta$ to be determined, we rewrite (\ref{1.6}) as
\begin{align}\label{3.1}
\widetilde{\mathcal A}_\tau(\bm u,\bm v):= {\mathcal A}_\tau(\bm u,\bm v)+\eta{\mathcal B}(\bm u,\bm v)=(\eta+\tau) {\mathcal B}(\bm u,\bm v),~\forall \bm v \in\mathbf{H}_0(\mathrm{curl}^2,D).
\end{align}

\begin{lem}Assume that $|\tau_1|>(\sqrt2-1)|\tau_2|$. For $\eta$ large enough,
$\widetilde{\mathcal A}_\tau$ is a coercive sesquilinear form on $\mathbf{H}_0(\mathrm{curl}^2,D)$.
\end{lem}
\begin{proof} Integrating by parts we have
\begin{align}\label{auxi}
\left((n-1)^{-1}\bm u, \mathrm{curl}^2 \bm u\right)&=\left(\mathrm{curl}((n-1)^{-1}\bm u), \mathrm{curl} \bm u\right)\nonumber\\
&=\left((n-1)^{-1}\mathrm{curl}\bm u+\nabla((n-1)^{-1})\times \bm u, \mathrm{curl} \bm u\right).
\end{align}
 Pick up any $\bm u\in\mathbf{H}_0(\mathrm{curl}^2,D)$. We have
\begin{align*}
|\widetilde{\mathcal A}_\tau(\bm u,\bm u)|&= \big|\left((n-1)^{-1}\mathrm{curl}^2 \bm u, \mathrm{curl}^2 \bm u\right)-\tau
    \left((n-1)^{-1}\bm u, \mathrm{curl}^2 \bm u\right)
 -\tau \left((n-1)^{-1}\mathrm{curl}^2 \bm u,\bm u \right)\\
 &~~~+\tau^2 \left(n(n-1)^{-1}  \bm u,\bm u \right)+\eta \left(\mathrm{curl} \bm u,\mathrm{curl}\bm u \right)\big|\\
 &\ge \big|\left((n-1)^{-1}\mathrm{curl}^2 \bm u, \mathrm{curl}^2 \bm u\right)+\tau^2 \left(n(n-1)^{-1}  \bm u,\bm u \right)+\eta \left(\mathrm{curl} \bm u,\mathrm{curl}\bm u \right)\big|\\
 &~~~-2\left|\tau\left((n-1)^{-1}\bm u, \mathrm{curl}^2 \bm u\right)\right|\\
                 &\ge\frac1{\sqrt2}\big |(N^*-1)^{-1}\|\mathrm{curl}^2\bm u\|^2+(\tau_1^2-\tau_2^2)(n(n-1)^{-1}\bm u,\bm u)+\eta \|\mathrm{curl} \bm u\|^2\big|\\
                 &~~~~+\sqrt2|\tau_1\tau_2|(n(n-1)^{-1}\bm u,\bm u)-2\left|\tau\left((n-1)^{-1}\bm u, \mathrm{curl}^2 \bm u\right)\right|
\end{align*}
The fact $|\tau_1|>(\sqrt2-1)|\tau_2|$ implies $2|\tau_1\tau_2|>\tau_2^2-\tau_1^2$.
Then for both $\tau_1^2>\tau_2^2$ and $\tau_1^2\le\tau_2^2$ it holds
\begin{align*}
|\widetilde{\mathcal A}_\tau(\bm u,\bm u)|
                 &\ge\frac1{\sqrt2}\big ((N^*-1)^{-1}\|\mathrm{curl}^2\bm u\|^2+\eta \|\mathrm{curl} \bm u\|^2\big)\\
                 &~~~~+(\sqrt2|\tau_1\tau_2|+(\tau_1^2-\tau_2^2)/\sqrt2)N^*(N^*-1)^{-1}(\bm u,\bm u)\\
                 &~~~~-2\left|(N_*-1)^{-1}\tau\left(\mathrm{curl}\bm u, \mathrm{curl} \bm u\right)\right|
                 -|(n-1)^{-1}|_{1,\infty}^2|\tau|^2\|\mathrm{curl}\bm u\|^2/\varepsilon- \varepsilon\|\bm u\|^2
\end{align*}
Hence we choose $\eta>\sqrt2\left(2\left|(N_*-1)^{-1}\tau\right|
                 +|(n-1)^{-1}|_{1,\infty}^2|\tau|^2/\varepsilon\right)$ with
$0<\varepsilon<(\sqrt2|\tau_1\tau_2|+(\tau_1^2-\tau_2^2)/\sqrt2)N^*(N^*-1)^{-1}$.
The assertion is valid.
\end{proof}

Hence we can define the solution operator $\widetilde T_\tau\bm f\in \mathbf{H}_0(\mathrm{curl}^2,D)\cap\b H(\mathrm{div},D)$ for $\mathrm{curl}\bm f\in\b L^2(D)$:
\begin{align}\label{1.10s}
\widetilde{\mathcal A}_{\tau}(\widetilde T_\tau\bm f,\bm v)= {\mathcal B}(\bm f,\bm v),~\forall \bm v  \in\mathbf{H}_0(\mathrm{curl}^2,D)
\end{align}
and its discrete operator
$\widetilde T_{\tau,h}\bm f\in \mathbf{U}_h$:
\begin{align}\label{1.10s1}
 \widetilde{\mathcal A}_{\tau}(\widetilde T_{\tau,h}\bm f,\bm v)=\mathcal B(\bm f,\bm v),~\forall \bm v  \in\mathbf{U}_h.
\end{align}
Like in Lemma 3.2, we know that $\widetilde T_\tau$ is compact in   
$\mathbf{H}_0(\mathrm{curl}^2,D)$.
Similar as in (\ref{1.16}), we can infer    the uniform convergence
\begin{align}\label{1.16t}
 \|\widetilde T_{\tau,h}-\widetilde T_\tau\|_{2,\mathrm{curl}}\rightarrow0.
\end{align}
Now we consider  the following  generalized eigenvalue problem
\begin{align}
 \widetilde{\mathcal A}_{\tau}(\bm u,\bm v)=(\widetilde\lambda(\tau)+\eta){\mathcal B}(\bm u,\bm v),~\forall \bm v  \in~\mathbf{H}_0(\mathrm{curl}^2,D).
\end{align}
and its discrete problem
\begin{align}\label{3.5}
\widetilde{\mathcal A}_{\tau}(\bm u_h,\bm v_h)=(\widetilde\lambda_h(\tau)+\eta) {\mathcal B}(\bm u_h,\bm v_h),~\forall \bm v _h \in\mathbf{H}_h.
\end{align}
Then $\widetilde\lambda(\tau)$ and $\widetilde\lambda_h(\tau)$ is a continuous function of $\tau$. From (\ref{3.1}), the Maxwell's transmission eigenvalue is the root of
$$\widetilde f(\tau):=\widetilde\lambda(\tau)-\tau$$
while the discrete eigenvalue in  (\ref{1.7}) is the root of
$$\widetilde f_h(\tau):=\widetilde\lambda_h(\tau)-\tau.$$
The error estimates of the discrete problem (\ref{3.5}) can be derived  like in Section 3. Hence we give this result without detailed proof.
\begin{lem}
    Let $\widetilde\lambda_h(\tau)$ be an eigenvalue of the  problem (\ref{3.5}) that converges to $\widetilde\lambda(\tau)$. Assume the ascent of $\widetilde\lambda(\tau)$ is one.
    Let the eigenfunction space $M(\widetilde\lambda(\tau))$ satisfy  $M(\widetilde\lambda(\tau))\subset \b H^{s-1}(D)$, $\mathrm{curl}\left(M(\widetilde\lambda(\tau))\right)\subset \b H^s(D)$ with $l+1\ge s>1$ then
\begin{eqnarray}
&&|\widetilde\lambda(\tau)-\widetilde\lambda_{h}(\tau)|\lesssim h^{2(s-1)}.
\end{eqnarray}
\end{lem}

\begin{theorem}
Let $\widetilde f_h(\tau)$ and $\widetilde f(\tau)$ be two analytic functions on a close convex domain $B$ not containing zero.
Let $\{\tau_{h_i}\}_1^{\infty}\subset B$ satisfy $\widetilde f_{h_i}(\tau_{h_i})=0$  then there is $\tau^*$ such that
$\widetilde f(\tau^*)=0$ and a subsequence  $\tau_{h_i}\to\tau^*$ ($i\to\infty$). We adopt the following assumption.

{\bf Assumption \~{A}.} There is a positive constant $c$ such that $\min_{\tau\in B}|\widetilde f'_{h}(\tau)|\ge c$ for $h$ small enough.

Then  it holds
\begin{align}\label{4.15a}
 |\tau_{h_i}-\tau^*|\le|\widetilde f_{h_i}(\tau^*)-\widetilde f(\tau^*)| {\overline c},~for~i~large~enough.
\end{align}

Conversely, let the domain $B$ be such that $\tau^*\in[a,b]$ with $\widetilde f(\tau^*)=0$ and   $\{\tau_{h}: \widetilde f_{h}(\tau_{h})=0,\forall h<\delta\}\subset B$ for a small $\delta>0$. If Assumption \~A  is valid
then any sequence $\{\tau_{h_i}\}_1^{\infty}\subset B$ satisfying $\widetilde f_{h_i}(\tau_{h_i})=0$ will converge to $\tau^*$  and the above (\ref{4.15a}) are valid.

\end{theorem}
\begin{proof}
The proof is similar to Theorem 3.1 except for the Lagrange mean value theorem
\begin{align*}
\widetilde f_{h_i}(\tau^*)-\widetilde   f_{h_i}(\tau_{h_i})=  (Re~\widetilde f'_{h_i}(\psi)+iIm~\widetilde f'_{h_i}(\eta))(\tau_{h_i}-\tau^*),\\
  \psi,\eta~are~on~the~line~between~\tau_{h_i}~and~
  \tau^*.
\end{align*}
\end{proof}

\begin{rem} In Theorem 4.1 the condition $\min_{\tau\in B}|\widetilde f'_{h}(\tau)|\ge c$ ($h$ small enough) can be reduced into $|\widetilde f'_{h}(\tau^*)|$ $\ge c$ ($h$ small enough).
We state the practicality of the conditions given in above theorem. The assumption $\{\tau_{h_i}\}_i^{\infty}\subset B$ is usually satisfied in the case that $\{\tau_{h_i}\}_i^{\infty}$ converges to a point $\tau^*$.  Assumption A  can be verified
when   $Re~\widetilde f^{}_{h_i}(\tau)$ or $Im~\widetilde f^{}_{h_i}(\tau)$  is a strictly monotonic function sequence along a line segment near and across $\tau_{h_i}$. By the same calculation  as in Remark 3.1
\begin{align}\label{3.222}
 \mathcal B(\bm u,\overline{\bm u}) \widetilde \lambda'(\tau^*)=-\left((N-I)^{-1}\mathrm{curl}^2 \bm u, \overline{\bm u}
\right)-\left((N-I)^{-1} \bm u,\mathrm{curl}^2 \overline{\bm u}
\right)+2\tau^*(N(N-I)^{-1}\bm u,\overline{\bm u})
 \end{align}
 with the eigenfunction   $\bm u:=\bm u(\tau^*)$ associated with $\widetilde \lambda(\tau^*)$.
Hence   $\lambda'(\tau^*)\ne1$ implies $|f'_{h_i}(\tau^*)|$ $\ge c$ ($i$ small enough).
\end{rem}

\begin{theorem}
Assume $\tau_h\to\tau^*\in(a,b)$ with $a>0$.  Let $\widetilde\lambda_h(\tau),\widetilde\lambda(\tau)\in C^1([a,b])$.   Then
\begin{eqnarray}
\label{2.36ss}
&&|\widetilde\lambda'_{h}(\tau_h)-\widetilde\lambda'(\tau^*)|+|\widetilde\lambda'_{h}(\tau^*)-\widetilde\lambda'(\tau^*)|\to0.
\end{eqnarray}
\end{theorem}

\begin{rem}
The above theorem implies that if $\tau_h\to\tau^*$ then $\widetilde f'_h(\tau_h)=\widetilde \lambda'_{h}(\tau_h)-1$ can be regarded as  an indicator to detect $|\widetilde f_h'(\tau^*)|=|\widetilde \lambda'_{h}(\tau^*)-1|$ greater than 0 strictly in Remark 3.
\end{rem}

The following theorem  is similar to Theorem 3.2 and thus we omit its proof.
\begin{theorem} Under the assumptions in Theorem 4.1,
    let $\tau_h=\widetilde \lambda_h(\tau_h)$ be an eigenvalue of the  problem (\ref{1.7}) that converges to $\tau^*=\widetilde \lambda(\tau^*)$ whose ascent is one.  Let the eigenfunction space $M(\widetilde \lambda(\tau^*))$ satisfy $M(\widetilde \lambda(\tau^*))\subset \b H^{s-1}(D)$, $\mathrm{curl}\left(M(\widetilde \lambda(\tau^*))\right)\subset \b H^s(D)$ with $l+1\ge s>1$ then
Let $\bm u_{h}(\tau_h)$ be the corresponding  eigenfunction of the problem (\ref{1.7})
  and $\|\bm u_{h}(\tau_h)\|_{2,\mathrm{curl}}=1$, then there exists an eigenfunction
$\bm u(\tau^*) $ and $s_0>1$ such that
\begin{eqnarray}
&&|\tau^*-\tau_h|\lesssim h^{2(s-1)},\\
&&\|\bm u_{h}(\tau_h)-\bm u_{}(\tau^*)\|_{2,\mathrm{curl}}\lesssim
h^{s-1} \\
&&\|\mathrm{curl}(\bm u_{h}(\tau_h)-\bm u_{}(\tau^*))\|_{}\lesssim
h^{s-2+s_0}.
\end{eqnarray}

\end{theorem}
%

\section{Numerical experiment}
In this section we shall show some numerical results to verify the condition $|f'_{h}(\tau_{})|\ge c$ in a  neighbourhood  of  $\tau_{h}$ for $h$ small enough in Theorems 3.1 and 4.1 (see also Remarks 3.1 and 4.1). For the case of the complex $\tau_{h}$, it   suffices to verify $|Im~f'_{h}(\tau)|\ge c$  in a  neighbourhood  of  $\tau_{h}$ for a small $h$.
 First of all, in order to  compute the eigenvalue problem (\ref{1.7}) we rewrite it as the linear formulation
\begin{align}\label{4.1}
    \begin{aligned}
    &\left((N-1)^{-1}\mathrm{curl}^2 \bm u_h , \mathrm{curl}^2 \bm v \right)= \tau_h
    \left((N-1)^{-1}\bm u_h, \mathrm{curl}^2 \bm v \right) \\
    &~~~~~~+ \tau_h \left(\mathrm{curl}^2 \bm u_h , N(N-I)^{-1}\bm v\right)- \tau_h  \left(\bm \omega_h, N(N-1)^{-1}\bm v\right),~~\forall \bm v\in \b U_h,\\
    &\left(\bm  \omega_h,
   \bm   z\right) = \tau_h \left(\bm u_h,\bm z\right),~~\forall \bm z\in \b U_h
    \end{aligned}
    \end{align}
which can be solved via direct eigenproblem solver. The second family of curl-curl    element with $l=3$ and the lowest order curl-curl element is adopted   to  solve the Maxwell's transmission eigenvalue
problem. The computed domain $D$ is chosen as
the unit square
$(0,1)^2$ or the L-shaped domain $(-1,1)^2\backslash \{[0,1)\times (-1,0]\}$.
The index of refraction $N$ is set to be the scalar-matrix $16I$ or $(8-x_1+x_2)I$ on the square and the L-shaped domains while set to be the  matrix
\begin{eqnarray*}
\left[
\begin{tabular}{cc}
16&0\\
 0&$16+x_1-x_2$
\end{tabular}
\right]
~or~
\left[
\begin{tabular}{cc}
16&$x_1$\\
 $x_1$&2
\end{tabular}
\right]
\end{eqnarray*}
on the L-shaped domain. The computed lowest four  eigenvalues obtained by the second family of curl-curl    element with $l=3$  are given in Tables 1-3 and those obtained by the lowest order curl-curl    element are given in Tables 4-6. It should be noted that all computed complex eigenvalues satisfy the assumption $Re~\tau>(\sqrt2-1)Im~\tau$ in Section 4.
In our computation, we take $h=\sqrt2/16$ on the square domain and $h=\sqrt2/8$ on the L-shaped domain to verify $|Im~f'_{h}(\tau)|\ge c$ in a small neighbourhood  of  $\tau_{h}$ (or equivalently $|Im~f'_{h}(\tau_h)|\ge c$). For this propose we choose a small neighbourhood of  a real number $\tau_h$  or a small line segment near and across a complex number $\tau_{h}$. According to the magnitude of eigenvalues, the neighbourhood is chosen between $(k_h-0.03)^2$ and $(k_h+0.03)^2$ and the line segment on the complex plane possesses the endpoints  $(k_h-0.03)^2$ and $(k_h+0.03)^2$. For example, if $k_h\approx 1.92$ then its neighbourhood is [$1.89^2,1.95^2$] while
if $k_h\approx 1.20+0.44i$ then the associated line segment is [$(1.17+0.44i)^2, (1.23+0.44i)^2$] on the complex plane. It can be seen from Figure 1 that for all computed cases $|f'_{h}(\tau_{})|$ is strictly greater than zero  in a  neighbourhood  of  $\tau_{h}$. In addition, according to Remark 2 and 4 we also adopt the formulas (\ref{3.221}) and (\ref{3.222}) to compute $f'_{h}(\tau_{})$ which are stable near fixed positive constant given in Tables 1-3. All of these numerical evidences indicate the conditions in Theorems 3.1 and 4.1 are valid. The computed convergence order $r_h$ of numerical eigenvalues on the square domain are around six, which is consistent with the theoretical results in Theorems 3.3 and 4.3. However, the convergence order computed by the second family of curl-curl element with $l = 3$ is much worse due to the singularities of the eigenvalue problem towards the L corner point. The   convergence order computed by the lowest order curl-curl element performs much well in this case.

At last it can be seen that all real numerical eigenvalues approximate the real eigenvalue from upper. In fact,
due to (\ref{t}) we have for the real eigenvalues $\tau_{h}$ and $\tau^*$
\begin{align}
  \tau_{h}-\tau^* =-(\lambda_{h}(\tau^*)-\lambda(\tau^*))/  f'_{h}(\psi).
\end{align}
By Lemma 9.1 in \cite{babuska}
\begin{eqnarray}\label{s4.8}
\lambda_{h}(\tau^*)-\lambda(\tau^*)=\frac{\mathcal A_{\tau_*}(\bm u_{h}(\tau^*)-\bm u(\tau^*),\bm u_{h}(\tau^*)-\bm u^{}(\tau^*))}{\mathcal B(\bm u_{h}(\tau^*),\bm u_{h}(\tau^*))}-\lambda(\tau^*)
\frac{\|\mathrm{curl}(\bm u_{h}(\tau^*)-\bm u(\tau^*))\|^2}{\mathcal B(\bm u_{h}(\tau^*),\bm u_{h}(\tau^*))}.
\end{eqnarray}
According to Theorem 3.3, the first term at the right-hand side is dominated so that $\lambda_{h}(\tau^*)>\lambda(\tau^*)$. This together with the fact $f'_{h}(\psi)<0$  shown in Fig. 1 yields the assertion $\tau_{h}>\tau^*$.

\begin{figure}
  \centering
  \includegraphics[width=2.5in]{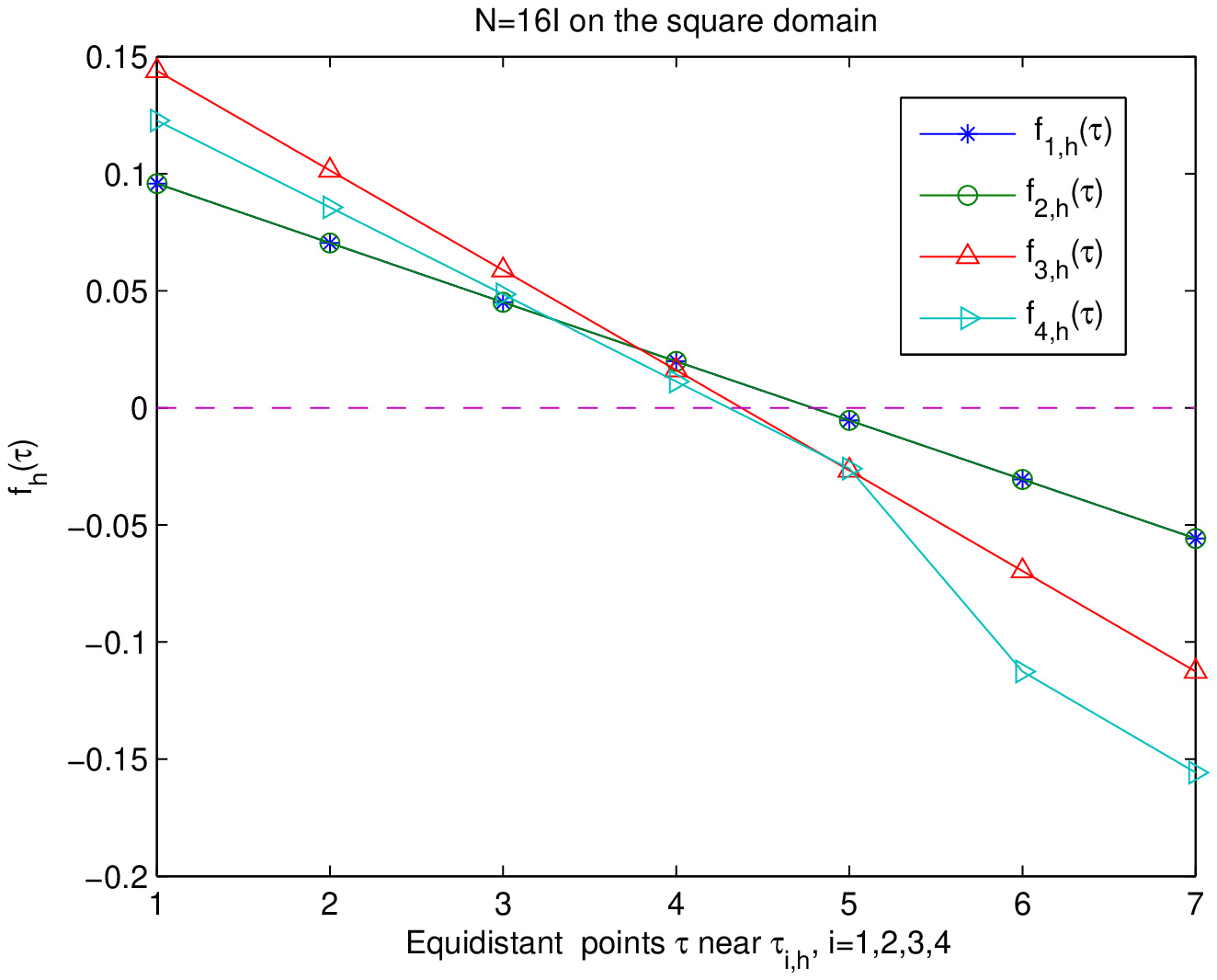}
  \includegraphics[width=2.5in]{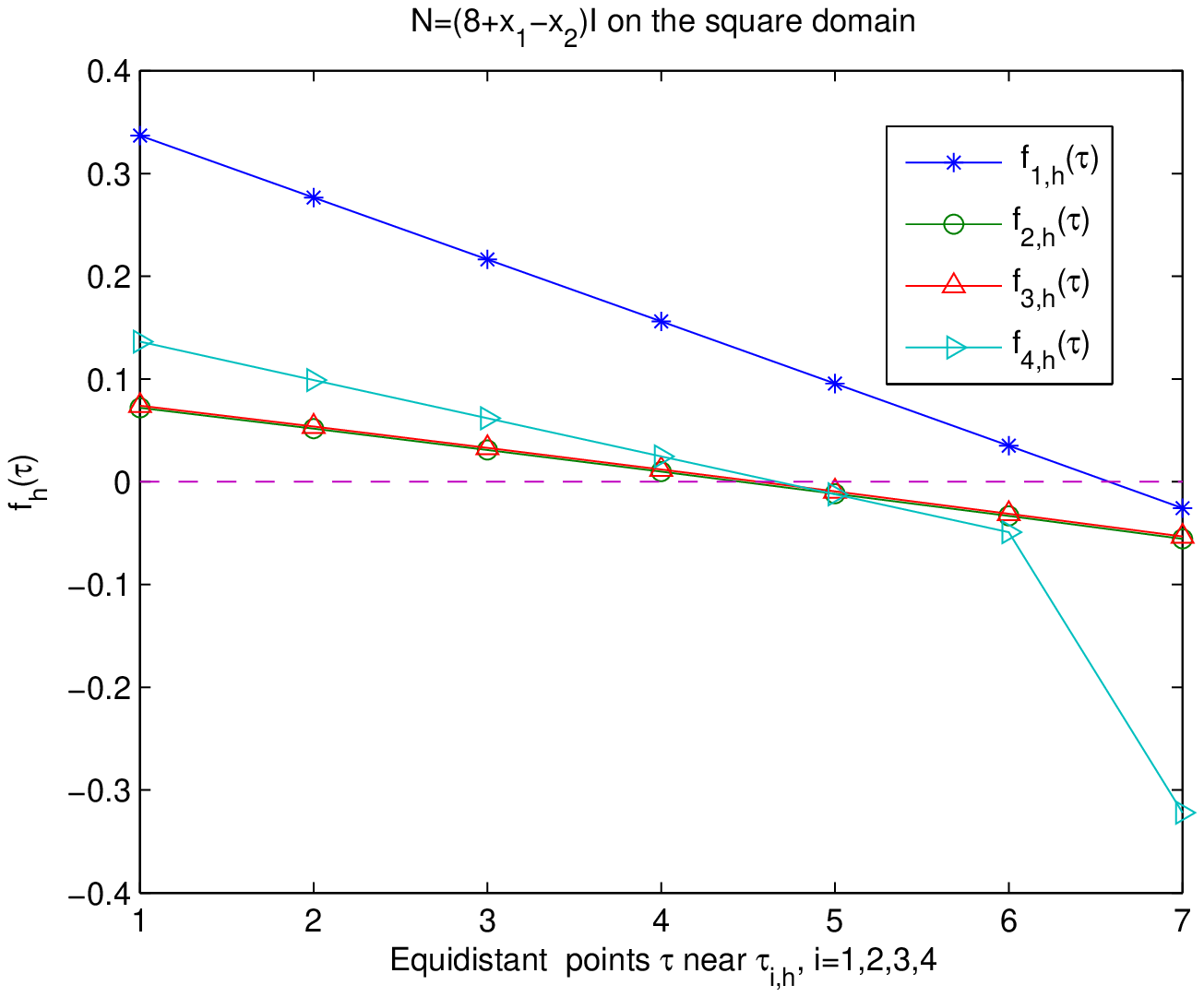}\\
  \includegraphics[width=2.5in]{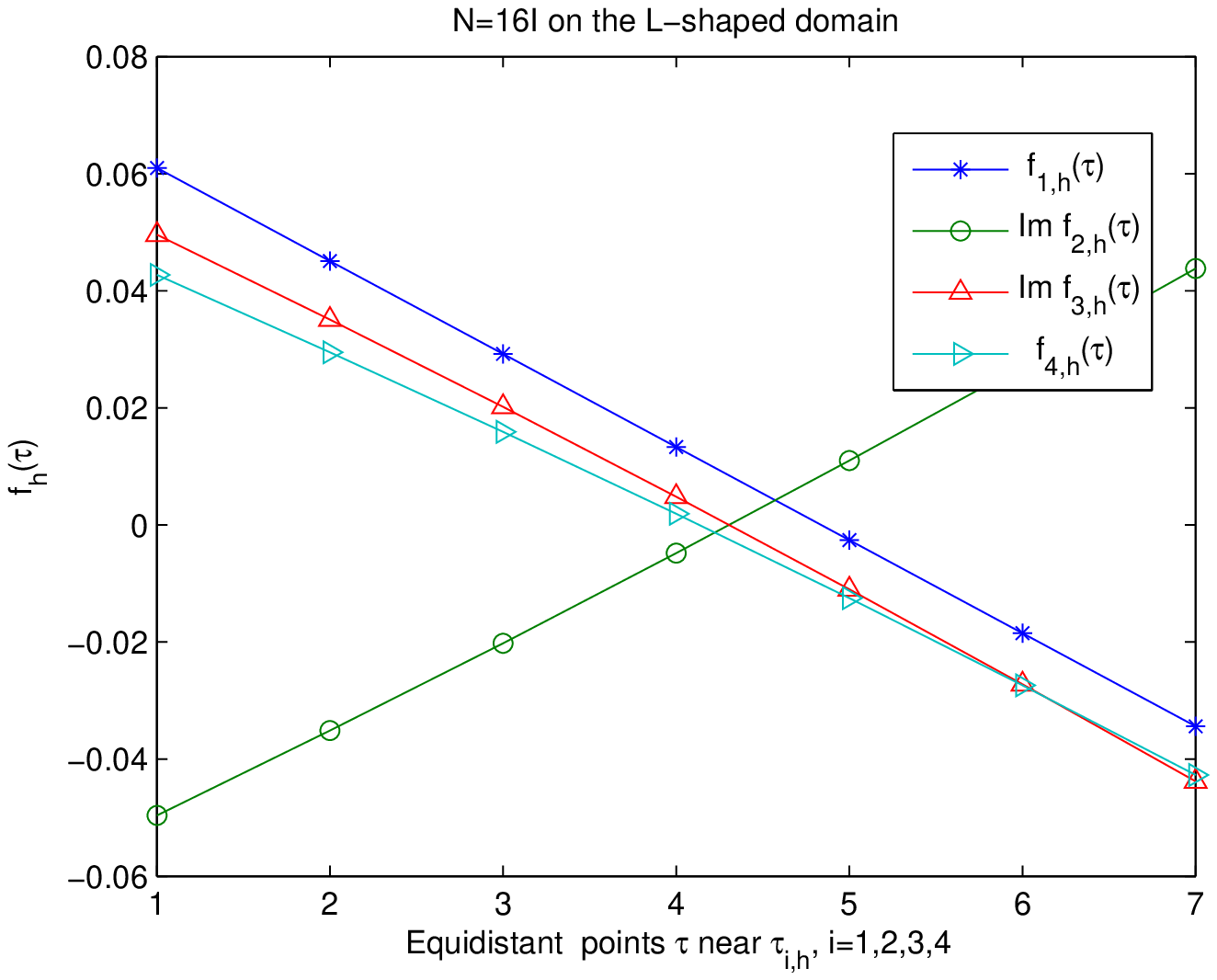}
  \includegraphics[width=2.5in]{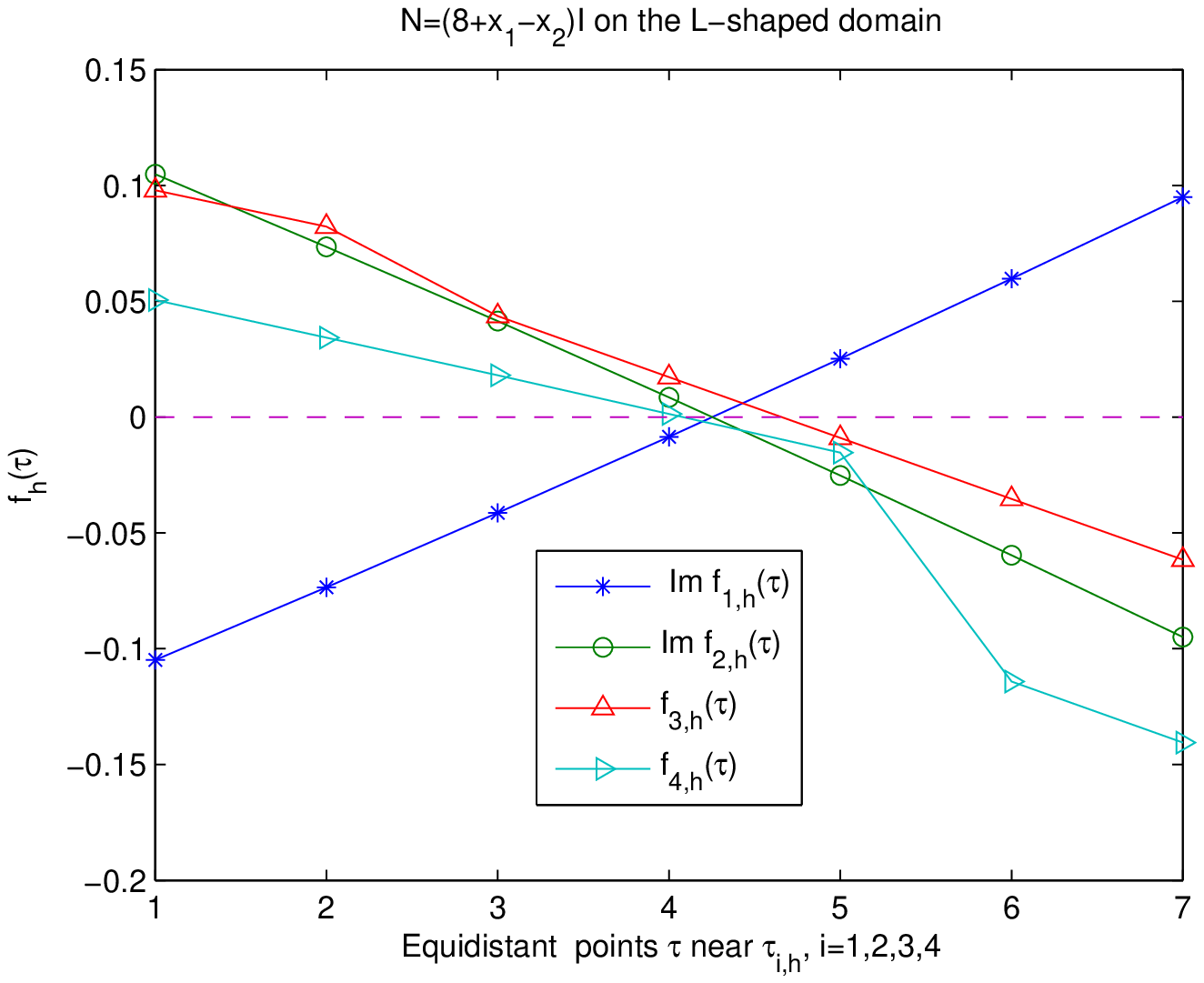}
  \includegraphics[width=2.5in]{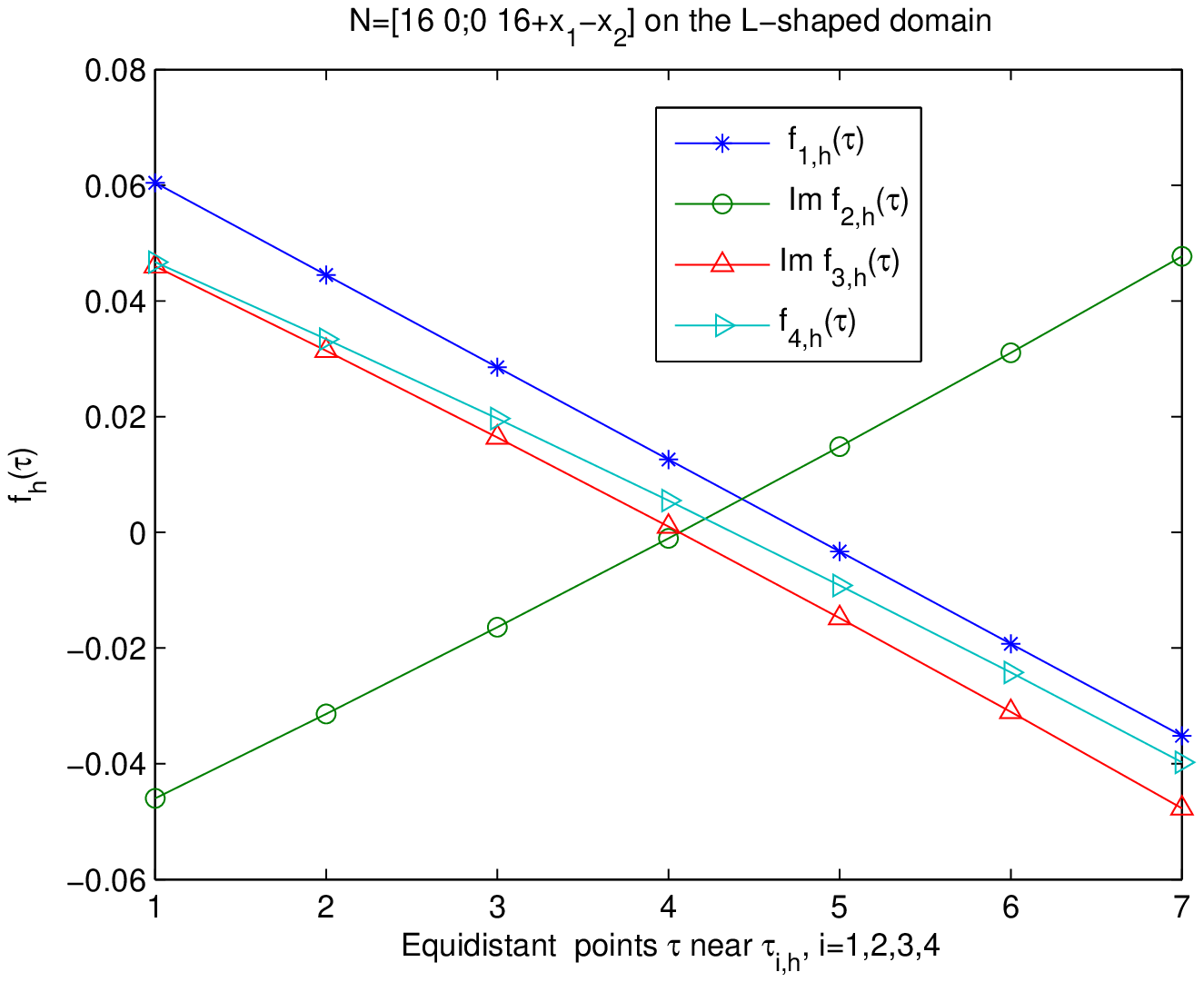}
  \includegraphics[width=2.5in]{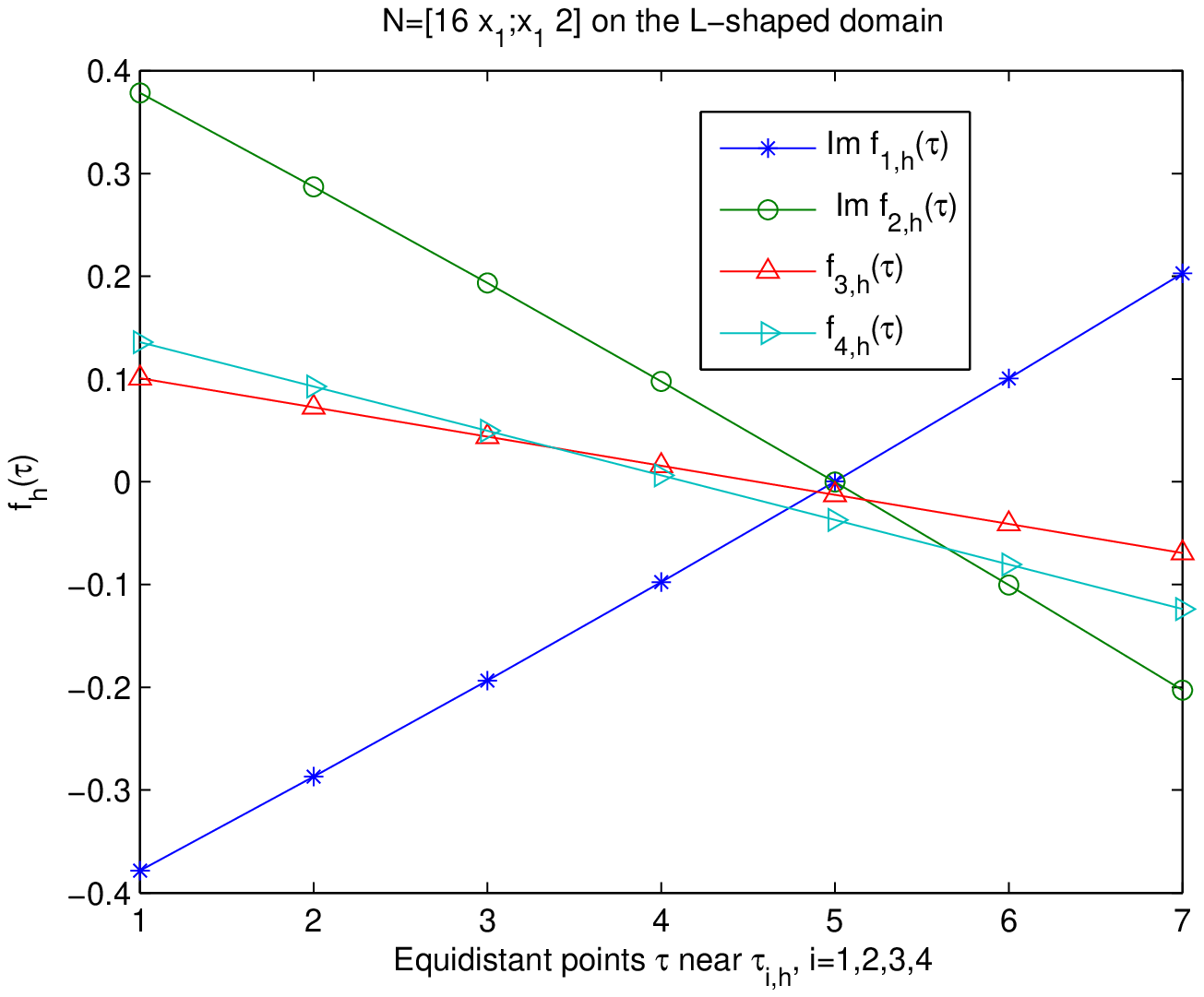}
  \caption{$f_h(\tau)$ computed by   the second family of curl-curl    element ($l=3$).}
\end{figure}

\begin{table}
\caption{Numerical eigenvalues by the second family of curl-curl element ($l=3$) on the unit square with
 $h_0=\frac{\sqrt2}{8}$.}
\begin{center} \footnotesize
\begin{tabular}{ccccc|ccccc}\hline
\multicolumn{5}{c|}{$N=16I$}&\multicolumn{3}{c}{$N=(8+x_1-x_2)I$}\\
$j$&$h$&$k_{j,h}$&$r_h$&$|f'_h(\tau_{j,h})|$&$k_{j,h}$&$r_h$&$|f'_h(\tau_{j,h})|$\\\hline
1&  $h_0$&1.927882001&&0.65&3.385891713&&0.80\\
1&  $\frac{h_0}2$& 1.927871544&5.14&0.65&3.385803303&5.92&0.80\\
1&  $\frac{h_0}4$&    1.927871248&&0.65&3.385801845&&0.80\\
2&  $h_0$&1.927882001&&0.65&3.505047783&&0.31\\
2&  $\frac{h_0}2$& 1.927871544&5.14&0.65&3.504614745&5.51&0.31\\
2&  $\frac{h_0}4$&    1.927871248&&0.65&3.504605216&&0.31\\
3&  $h_0$&    2.333811701&&0.92&3.505832045&&0.31\\
3&  $\frac{h_0}2$&2.333763623&5.78&0.92&3.505632716&4.82&0.31\\
3&  $\frac{h_0}4$&2.333762748&&0.92&3.505625649&&0.31\\
4&  $h_0$&2.343100969&&0.79&3.616924954&&0.51\\
4&  $\frac{h_0}{2}$&   2.343034872&5.75&0.79&3.616668920&5.37&0.51\\
4&  $\frac{h_0}{4}$& 2.343033643&&0.79&3.616662744&&0.51
\\\hline
\end{tabular}
\end{center}
\end{table}

\begin{table}
\begin{center} \footnotesize
\caption{Numerical eigenvalues by the second family of curl-curl element ($l=3$) on the L-shaped domain with
 $h_0=\frac{\sqrt2}{4}$.}
\begin{tabular}{ccccc|ccccc}\hline
\multicolumn{5}{c|}{$N=16I$}&\multicolumn{4}{c}{$N=(8+x_1-x_2)I$}\\
$j$&$h$&$k_{j,h}$&$r_h$&$|f'_h(\tau_{j,h})|$&$j$&$h$&$k_{j,h}$&$r_h$&$|f'_h(\tau_{j,h})|$\\\hline
1&  $h_0$& 1.17856958&&0.67&1,2&$h_0$&1.2921312&&1.90\\
1&  $\frac{h_0}{2}$&1.1783712&1.50&0.67&&&$\pm$0.6797234i\\
1&  $\frac{h_0}{4}$&   1.1783009&&0.67&1,2&$\frac{h_0}{2}$&1.2925357&1.31&1.89\\
3,4&  $h_0$&1.2025191&&0.96&&&$\pm$0.6797410i\\
 &  &  $\pm$0.4412079i&&&1,2& $\frac{h_0}{4}$&1.2926992&&1.89\\
3,4&  $\frac{h_0}{2}$&1.2030798&1.41&0.96&&&$\pm$0.6797388i\\
&& $\pm$0.4411595i&&&3&$h_0$&2.0370694&&0.65\\
3,4&  $\frac{h_0}{4}$& 1.2032896&&0.96&3&$\frac{h_0}{2}$&2.0365801&2.35&0.65\\
 && $\pm$0.4411257i&&&3&$\frac{h_0}{4}$&2.0364839&&0.65\\
2&  $h_0$& 1.2717694&&0.56&4&$h_0$&2.0630459&&0.40\\
2&  $\frac{h_0}{2}$&1.2713493&1.76&0.56&4&$\frac{h_0}{4}$&2.0608619&1.42&0.40\\
2&  $\frac{h_0}{4}$&   1.2712256&&0.56&4&$\frac{h_0}{4}$&2.0600483&&0.40\\
\hline
\end{tabular}
\end{center}
\end{table}

\begin{table}
\begin{center} \footnotesize
\caption{Numerical eigenvalues by the second family of curl-curl element ($l=3$)  on the L-shaped-domain with $h_0=\frac{\sqrt2}{4}$.}
\begin{tabular}{ccccc|ccccc}\hline
\multicolumn{5}{c|}{$N=[16~0;0~16+x_1-x_2]$}&\multicolumn{4}{c}{$N=[16~x_1;x_1~2]$}\\
$j$&$h$&$k_{j,h}$&$r_h$&$|f'_h(\tau_{j,h})|$&$j$&$h$&$k_{j,h}$&$r_h$&$|f'_h(\tau_{j,h})|$\\\hline
1&  $h_0$&1.188137&&0.67&1,2&  $h_0$&  1.5794070&& 4.67\\
1&  $\frac{h_0}{2}$&1.187914      &   1.47&0.67&&&    $\pm$1.0033762i\\
1&  $\frac{h_0}{4}$& 1.187833&&0.67&1,2&  $\frac{h_0}{2}$&1.5799704&0.89& 4.67\\
2,3&  $h_0$& 1.2000971&&0.97 && &     $\pm$1.0032588i\\
&& $\pm$0.4413211i&&&1,2&  $\frac{h_0}{4}$&     1.5802770&& 4.67\\
2,3&  $\frac{h_0}{2}$&1.2006542&1.40&0.96&& &           $\pm$1.0032045i\\
 & &     $\pm$0.4412676i&&&3&  $h_0$&  2.1671018&& 0.65\\
2,3&  $\frac{h_0}{4}$&    1.2008635&&0.96&3&  $\frac{h_0}{2}$& 2.1654994      &   2.08& 0.65\\
 &   &        $\pm$0.4412319i&&&3&  $\frac{h_0}{4}$&   2.1651206&& 0.65\\
4&  $h_0$&      1.2842037&&0.57&4&  $h_0$&     2.5257509&&  0.86\\
4&  $\frac{h_0}{2}$&1.2837800&1.77&0.57&4&  $\frac{h_0}{2}$&2.5214581&2.09&  0.86\\
4&  $\frac{h_0}{4}$& 1.2836553&& 0.57&4&  $\frac{h_0}{4}$& 2.5204517&&  0.86
\\\hline
\end{tabular}
\end{center}
\end{table}

\begin{table}
\caption{The numerical eigenvalues by lowest order element  on the unit square with
 $h_0=\frac{\sqrt2}{16}$.}
\begin{center} \footnotesize
\begin{tabular}{ccccc|ccccc}\hline
\multicolumn{5}{c|}{$N=16I$}&\multicolumn{3}{c}{$N=(8+x_1-x_2)I$}\\
$j$&$h$&$k_{j,h}$&$r_h$&$|f'_h(\tau_{j,h})|$&$k_{j,h}$&$r_h$&$|f'_h(\tau_{j,h})|$\\\hline
1&  $h_0$&1.9556&&0.64&3.4855&&0.77\\
1&  $\frac{h_0}2$&1.9348&2.00&0.65&3.4116&1.94&0.79\\
1&  $\frac{h_0}4$&1.9296&2.00&0.65&3.3923&1.98&0.80\\
1&  $\frac{h_0}8$&1.9283&&0.65&3.3874&&0.80\\
2&  $h_0$&1.9773&&0.62&3.6570&&0.32\\
2&  $\frac{h_0}2$&1.9400&2.04&0.65&3.5430&2.02&0.31\\
2&  $\frac{h_0}4$&1.9309&1.98&0.65&3.5149&2.01&0.31\\
2&  $\frac{h_0}8$&1.9286&&0.65&3.5079&&0.31\\
3&  $h_0$&2.3787&&0.86&3.8035&&0.39\\
3&  $\frac{h_0}2$&2.3465&1.81&0.89&3.5885&1.76&0.31\\
3&  $\frac{h_0}4$&2.3373&1.82&0.91&3.5252&2.03&0.31\\
3&  $\frac{h_0}8$&2.3347&&0.92&3.5097&&0.31\\
4&  $h_0$&2.4270&&0.80&3.8496&&0.45\\
4&  $\frac{h_0}{2}$&2.3624&2.12&0.80&3.6598&2.54&0.50\\
4&  $\frac{h_0}{4}$&2.3475&2.13&0.79&3.6272&2.04&0.51\\
4&  $\frac{h_0}{8}$&2.3441&&0.79&3.6193&&0.51
\\\hline
\end{tabular}
\end{center}
\end{table}

\begin{table}
\begin{center} \footnotesize
\caption{The numerical eigenvalues by lowest order element  on the L-shaped domain with
 $h_0=\frac{\sqrt2}{16}$.}
\begin{tabular}{ccccc|ccccc}\hline
\multicolumn{5}{c|}{$N=16I$}&\multicolumn{4}{c}{$N=(8+x_1-x_2)I$}\\
$j$&$h$&$k_{j,h}$&$r_h$&$|f'_h(\tau_{j,h})|$&$j$&$h$&$k_{j,h}$&$r_h$&$|f'_h(\tau_{j,h})|$\\\hline
1&  $h_0$& 1.1887&&0.66&1,2&$h_0$&1.2911&&1.91\\
1&  $\frac{h_0}{2}$&1.1810&1.94&0.67&&&$\pm$0.6823i\\
1&  $\frac{h_0}{4}$&1.1790&1.86&0.67&1,2&$\frac{h_0}{2}$&1.2919&1.54&1.90\\
1&  $\frac{h_0}{8}$&1.1784&&0.67&&&$\pm$0.6804i&&\\
3,4&  $h_0$&1.1992&&0.97&1,2&$\frac{h_0}{4}$&1.2924&1.41&1.90\\
 &  &  $\pm$0.4436i&&&&&$\pm$0.6799i&&\\
3,4&  $\frac{h_0}{2}$&1.2018&1.50&0.97&1,2&$\frac{h_0}{8}$&1.2926&&1.89\\
&& $\pm$0.4418i&&&&&$\pm$0.6798i&&\\
3,4&  $\frac{h_0}{4}$& 1.2028&1.48&0.97&3&$h_0$&2.0675&&0.63\\
 && $\pm$0.4413i&&&3&$\frac{h_0}{2}$&2.0443&2.00&0.64\\
3,4&  $\frac{h_0}{8}$& 1.2032&&0.96&3&$\frac{h_0}{4}$&2.0385&1.93&0.64\\
 && $\pm$0.4412i&&&3&$\frac{h_0}{8}$&2.0370&&0.65\\
2&  $h_0$&1.2880&&0.57&4&$h_0$&2.1037&&0.40\\
2&  $\frac{h_0}{2}$&1.2757&1.90&0.56&4&$\frac{h_0}{2}$&2.0716&1.90&0.40\\
2&  $\frac{h_0}{4}$&1.2724&1.89&0.56&4&$\frac{h_0}{4}$&2.0630&1.82&0.40\\
2&  $\frac{h_0}{8}$&1.2715&&0.56&4&$\frac{h_0}{8}$&2.0606&&0.40\\
\hline
\end{tabular}
\end{center}
\end{table}

\begin{table}
\begin{center} \footnotesize
\caption{The numerical eigenvalues by lowest order element on the L-shaped-domain with $h_0=\frac{\sqrt2}{16}$.}
\begin{tabular}{ccccc|ccccc}\hline
\multicolumn{5}{c|}{$N=[16~0;0~16+x_1-x_2]$}&\multicolumn{4}{c}{$N=[16~x_1;x_1~2]$}\\
$j$&$h$&$k_{j,h}$&$r_h$&$|f'_h(\tau_{j,h})|$&$j$&$h$&$k_{j,h}$&$r_h$&$|f'_h(\tau_{j,h})|$\\\hline
1&  $h_0$&1.1984&&0.66&1,2&  $h_0$&  1.5803&& 4.69\\
1&  $\frac{h_0}{2}$&1.1906&   1.89&0.67&&&    $\pm$1.007i\\
1&  $\frac{h_0}{4}$&1.1885&2.02&0.67&1,2&  $\frac{h_0}{2}$&1.5798&1.83& 4.67\\
1&  $\frac{h_0}{8}$& 1.1880&&0.67 && &     $\pm$1.004i\\
2,3&  $h_0$& 1.1968&&0.97 &1,2&  $\frac{h_0}{4}$&    1.5801&1.59&4.67\\
&& $\pm$0.4438i&&&&&    $\pm$1.004i && \\
2,3&  $\frac{h_0}{2}$&1.1994&1.44&0.96&1,2&  $\frac{h_0}{8}$&          1.5803&&4.67\\
 & &     $\pm$0.4420i&&&&&  $\pm$1.003i&& \\
2,3&  $\frac{h_0}{4}$&    1.2004&1.65&0.96&3&  $h_0$& 2.2286      &   1.98& 0.65\\
 &   &        $\pm$0.4414i&&&3&  $\frac{h_0}{2}$&   2.1812&& 0.65\\
2,3&  $\frac{h_0}{8}$&    1.2008&&0.96&3&  $\frac{h_0}{4}$& 2.1692     &   1.95& 0.65\\
 &   &        $\pm$0.4413i&&&3&  $\frac{h_0}{8}$&   2.1661&& 0.65\\
4&  $h_0$&      1.3006&&0.57&4&  $h_0$&     2.6139&&  0.83\\
4&  $\frac{h_0}{2}$&1.2881&1.92&0.57&4&  $\frac{h_0}{2}$&2.5446&1.94&  0.85\\
4&  $\frac{h_0}{4}$& 1.2848&1.95& 0.57&4&  $\frac{h_0}{4}$& 2.5266&1.91&  0.86\\
4&  $\frac{h_0}{8}$& 1.2839&& 0.57&4&  $\frac{h_0}{8}$& 2.5218&&  0.86
\\\hline
\end{tabular}
\end{center}
\end{table}

\indent{\bf Acknowledgements.}~~
This work  was partially  supported by
 supported by the National Natural Science Foundation of China (Grants. 12001130, 11871092, NSAF 1930402),
China Postdoctoral Science Foundation (no. 2020M680316), and Science and Technology Foundation of Guizhou Province
(no. ZK[2021]012).

\end{document}